\documentclass{article}
\usepackage{graphicx} 

\usepackage{geometry}
\usepackage{ dsfont }
\usepackage[utf8]{inputenc}
\usepackage{amsmath,amssymb,amsthm}
\usepackage[all]{xy} 
\usepackage{ mathrsfs }
\usepackage{ stmaryrd }
\usepackage{ upgreek }
\usepackage{comment}

\usepackage{fancyhdr}

\usepackage{setspace}

\usepackage{dirtytalk}

\usepackage{faktor}

\usepackage[shortlabels]{enumitem}

\usepackage{xcolor}

\usepackage[hyperfootnotes=false]{hyperref}

\hypersetup{
  colorlinks=true,
  linkcolor=blue,
  linkbordercolor=blue,
  citebordercolor=green
}

\usepackage{tikz}
\usetikzlibrary{shapes.geometric,fit}
\usetikzlibrary{arrows.meta}
\usepackage{tkz-euclide}
\usepackage{tikz-cd}

\usetikzlibrary{calc,trees,positioning,arrows,chains,shapes.geometric,shapes,shadows,matrix}
\usepackage{blkarray}
\usepackage{wrapfig}
\usetikzlibrary{tikzmark}
\usepackage{float}
\usepackage{graphicx}

\usepackage{array}
\newcolumntype{C}{>{\centering\arraybackslash}p{1.2cm}}

\usepackage{url}

\usepackage{comment}

\usepackage{esvect} 

\usepackage{caption} 

\usepackage{longtable}

\usepackage{changepage}

\usepackage{biblatex}
\newtheorem{theorem}{Theorem}[section]
\newtheorem{lemma}[theorem]{Lemma}
\newtheorem{proposition}[theorem]{Proposition}

\newtheorem{corollary}[theorem]{Corollary}

\newtheorem{question}{Question}

\theoremstyle{definition}

\newtheorem{example}[theorem]{Example}

\newtheorem*{remark}{Remark}

\newtheoremstyle{indentclaim}
  {}          
  {}          
  {}          
  {2em}       
  {\itshape}  
  {:}         
  {.5em}      
  {}          

\theoremstyle{indentclaim}
\newtheorem*{claim}{Claim}

\newtheoremstyle{named}{}{}{\itshape}{}{\bfseries}{.}{.5em}{\thmnote{#3's }#1}
\theoremstyle{named}

\newtheoremstyle{named}{}{}{\itshape}{}{\bfseries}{.}{.5em}{\thmnote{#3 }#1}
\theoremstyle{named}
\newtheorem*{namedconjecture}{Conjecture}

\newcommand{\mainthmname}{}
\newtheorem*{mainthminner}{\mainthmname}
\newenvironment{mainthm}[1]
  {\renewcommand{\mainthmname}{#1}\begin{mainthminner}}
  {\end{mainthminner}}

\title{Edge-coloring Power Graphs}
\author{}
\date{Elie Feinsilber}

\begin{document}

\maketitle

\begin{abstract}
In this paper, we investigate the edge-coloring number of various graphs defined on finite groups. We notably show that the power graph of a finite group $G$ is overfull if and
only if the power graph of $G$ is of Class $2$ (has edge-coloring number one more
than its maximum vertex degree) if and only if $G$ is a cyclic group of odd prime
power order. We find the equivalent necessary and sufficient conditions to be Class $2$ for the intersection power graph and enhanced power graph (respectively cyclic of odd prime power order, and cyclic of odd order), and investigate the situation for the proper power graph. Finally, we give conditions for the power graph of a finite group to have the universal edge pre-coloring extension property: every optimal coloring of the power graph extending to another graph without modifying the existing colors (the enhanced power graph in our case). 
\end{abstract}

\section{Introduction}

The study of graphs defined on groups has been a growing area of research in the last decades. The study of these is not new as Cayley graphs have been around for some time, but their study was brought a new perspective when Kelarev and Quinn \cite{KelarevQuinn} introduced the \textit{directed power graph} in 2002, a directed graph with vertex set the elements of a group and where $x \rightarrow y$ if $y$ is a power of $x$. Motivated by this, Chakrabarty, Ghosh, and Sen \cite{C-G-S} introduced the \textit{power graph}: the simple graph $\mathscr{G}(G)$ with vertex set the elements of $G$ and where $a \sim b$ if and only if one is a power of the other. 

Several variants of the power graph emerged, such as the \textit{proper power graph}, first studied in \cite{MoghaddamfarRahbariyanShi}, denoted by $\mathscr{G}^{*}(G)$, where the identity of $G$ is removed from $\mathscr{G}(G)$. The \textit{enhanced power graph}, introduced in \cite{AalipourEtAl2017}, and denoted by $\mathscr{G}_{e}(G)$, has vertex set $G$, and two distinct vertices $x$ and $y$ are adjacent if and only if they belong to a common cyclic subgroup of $G$, or equivalently, if $\langle x,y\rangle$ is cyclic. Finally, the \textit{intersection power graph} introduced in \cite{Bera2018} and denoted by $\mathscr{G}_{I}(G)$, also has vertex set $G$ and two distinct vertices $x,y\in G$ are adjacent if and only if $\langle x\rangle\cap\langle y\rangle\neq\{e\}$, or equivalently, if some non-identity element of $G$ is a power of both $x$ and $y$ (and the identity is adjacent to all other vertices). 

Such graphs defined on algebraic structures have been studied extensively in recent years. A good overview of the results and current progress can be found in \cite{PJC3}, \cite{Grazian}, or \cite{Survey-PG-groups}. 

\medskip

Remark that some graphs are  subgraphs of each other. Notably, \[\mathscr{G}^*(G) \subseteq \mathscr{G}(G) \subseteq \mathscr{G}_e(G) \text{ and } \mathscr{G}(G) \subseteq \mathscr{G}_I(G)\] 
where $\mathscr{G}(G)$ is a spanning subgraph of both $\mathscr{G}_e(G)$ and $\mathscr{G}_I(G)$ although the enhanced power graph and intersection power graph are incomparable. 

\bigskip

One of the most natural graph-theoretic questions that can be asked of graphs, and has many applications, is to determine its edge-coloring numbers. This is especially applicable to power graphs given that they often contain large cliques and dense subgraphs. We will see that the edge-coloring detects a clear algebraic structure of the underlying group. 

We use $\chi'(\Gamma)$ to denote the minimum number of colors required to edge-color a graph $\Gamma$ (i.e., such that no two adjacent edges have the same color). It was shown by Vizing \cite{Vizing} and Gupta \cite{Gupta} independently that, for any graph, its coloring number is either the maximum degree of any vertex in the graph, or one more. We denote by $\Delta(\Gamma)$ the largest degree of any vertex in the graph $\Gamma$. 

\begin{theorem}(Vizing, Gupta)
    For every graph $\Gamma$, $\Delta(\Gamma) \leq \chi'(\Gamma) \leq \Delta(\Gamma)+1$. 
\end{theorem}

Following this theorem, graphs with $\chi'(\Gamma) = \Delta(\Gamma)$ are referred to as Class 1 and graphs with $\chi'(\Gamma) = \Delta(\Gamma)+1$ are referred to as Class 2. We remark that determining whether a graph is Class 1 is an NP-complete problem (Holyer in \cite{NP}). 

\medskip

A few tools seem to have been very useful in the study of edge-colorings of graphs. First, the notion of ``overfullness" was introduced by Chetwynd and Hilton in \cite{C-H-initial}, although they only defined the exact term later in \cite{C-H-conjecture}. Given a graph $\Gamma$ on $n$ vertices, we say that $\Gamma$ is \textit{overfull} if $|E(\Gamma)| / \lfloor\frac{n}{2}\rfloor > \Delta(\Gamma)$. If $\Gamma$ is overfull, then $\chi'(\Gamma) = \Delta(\Gamma)+1$ as no color can be used for more than $\lfloor\frac{n}{2}\rfloor$ edges. 
Notice that an overfull graph necessarily has odd order as, if $n$ is even, then the general bound
$|E(\Gamma)|\leq n\Delta(\Gamma)/2
=\Delta(\Gamma)\left\lfloor\frac n2\right\rfloor$
forbides overfullness. Thus, when $n$ is odd, the condition becomes $|E(\Gamma)|>(n-1)\Delta(\Gamma)/2$.
In this sense, an overfull graph is necessarily close to being $\Delta$($\Gamma$)-regular. 

One other tool for looking at edge colorings is what is called the core of a graph. The \textit{core} of $\Gamma$, denoted by $\Gamma_\Delta$, is the subgraph of $\Gamma$ induced by the vertices of degree $\Delta(\Gamma)$. This core has implications on the class of a graph:

\begin{theorem}Vizing \cite{Vizing-core}\label{Vizing}:
    If $\Gamma_\Delta$ has at most two vertices, then $\Gamma$ is Class 1. 
\end{theorem}

\begin{theorem}Fournier \cite{Fournier}\label{fournier}:
    If $\Gamma_\Delta$ contains no cycle, then $\Gamma$ is Class 1. 
\end{theorem}

Finally, Chetwynd and Hilton made a conjecture about edge colorings: 

\begin{namedconjecture}[Overfull]Chetwynd and Hilton \cite{C-H-conjecture}:
    A graph \( \Gamma \) with \( \Delta(\Gamma) > \frac{n}{3} \) is Class 2 if and only if it has an overfull subgraph \( S \) such that \( \Delta(\Gamma) = \Delta(S) \).
\end{namedconjecture}

Some progress was made on the conjecture in various papers by Chetwynd and Hilton \cite{C-H-3}, \cite{C-H_other}, Plantholt \cite{Plantholt-1}, \cite{Plantholt_even}, and subsequently Rhee \cite{Rhee}, the combination of which amounted to the following theorem: 
 
\begin{theorem}\label{overfull proven part}
    Let $\Gamma$ be a simple graph with $\Delta(\Gamma) \geq |V(\Gamma)| - 3$. Then $\Gamma$ is Class $2$ if and only if $\Gamma$ contains an overfull subgraph $H$ with $\Delta(H) = \Delta(\Gamma)$. 
\end{theorem}

\bigskip

Back to power graphs, since the identity element of the group $G$ is a power of every element of the group, we can get a clearer idea of the structure of the graph. Notably, the following result of Cameron is very useful to us. 

\begin{proposition}\cite{PJC2}\label{non trivial center implications}
     \textit{Let $G$ be a finite group; let $S$ be the set of vertices of the power graph $\mathscr{G}(G)$ which are joined to all other vertices. Suppose that $|S| > 1$. Then one of the following occurs:}
     \begin{enumerate}[(a)]
    \item $G$ is cyclic of prime power order, and $S = G$;
    \item $G$ is cyclic of non-prime-power order $n$, and $S$ consists of the identity and the generators of $G$, so that $|S| = 1 + \varphi(n)$;
    \item $G$ is generalised quaternion, and $S$ contains the identity and the unique involution in $G$, so that $|S| = 2$.
    \end{enumerate}
\end{proposition}


\subsection{Main Results} 
 
Our results show that edge-coloring, although difficult in general, becomes rigid enough on (some) graphs defined on groups to recover precise algebraic classifications; comparing the edge chromatic number of these constructions then naturally leads from chromatic-index questions to a stronger universal extension problem in which group-theoretic tools allow a different approach. 

In the power graph, the enhanced power graph and the intersection power
graph of a group $G$, the identity is adjacent to every other vertex. Each of these graphs 
therefore has maximum degree $|G|-1$, and by the result of Vizing and Gupta
its chromatic index is either $|G|-1$ or $|G|$. Our first result determines
which case occurs for the power graph.
 
\begin{mainthm}{Theorems~\ref{ovf iff} and~\ref{thm:pg-class2}}\label{thmA}
Let $G$ be a finite group of order $n\geq 2$. The following are equivalent:
\begin{enumerate}
  \item[(i)] $\mathscr{G}(G)$ is overfull;
  \item[(ii)] $\mathscr{G}(G)$ is Class~$2$, that is, $\chi'(\mathscr{G}(G))=n$;
  \item[(iii)] $G$ is cyclic of odd prime-power order.
\end{enumerate}
\end{mainthm}
 
Equivalently, $\mathscr{G}(G)$ is Class~$2$ exactly when it is a complete
graph of odd order. The proof combines Cameron's description of the vertices
of $\mathscr{G}(G)$ joined to all others
(Proposition~\ref{non trivial center implications}) with the solved case of
the Overfull Conjecture (Theorem~\ref{overfull proven part}). Together they
reduce the problem to counting the edges missing from $\mathscr{G}(G)$. We give an example in the Appendix. 
 
The same approach settles the question for the enhanced power graph and the
intersection power graph in Theorems \ref{thm:enhanced-power} and \ref{thm:intersection-power} respectively, both of which contain $\mathscr{G}(G)$ as a
spanning subgraph.

\medskip

We thus know when chromatic indices of these three graphs differ or agree: $\chi'(\mathscr{G}(G)) = \chi'(\mathscr{G}_I(G))$ always holds, and $\chi'(\mathscr{G}(G)) = \chi'(\mathscr{G}_e(G))$ holds except when $G$ is cyclic of odd order but not a prime power. This leads to a robustness question about the relation between colorings of graphs defined on a same group. Since $\mathscr{G}(G)$ is a spanning subgraph
of $\mathscr{G}_{e}(G)$, one can ask whether every optimal edge-coloring of
$\mathscr{G}(G)$ can be completed to an edge-coloring of $\mathscr{G}_{e}(G)$
without recoloring any edge (this question is also applicable using the intersection power graph, but we do not consider it here). This is a universal version of the
edge-precoloring extension problem~\cite{GiraoKang2019}. We call this the \emph{universal edge-precoloring extension property} (UEPE). The UEPE holds trivially
when $\mathscr{G}(G)=\mathscr{G}_{e}(G)$, that is, when every cyclic subgroup
of $G$ has prime-power order~\cite[Theorem~28]{AalipourEtAl2017}. It fails
whenever the two graphs have different chromatic indices. 

In section \ref{sec:uepe} we give conditions for a graph to have the UEPE based on degree in Theorem \ref{thm:criterion}, which implies that every finite centerless group and every finite dihedral group have the UEPE (Corollary \ref{cor:centerless-dihedral}). We also show the following results by relying on a common obstructive construction:  

\begin{mainthm}{Theorem \ref{thm:cyclic-exact}}
    The cyclic group $C_n$ has the UEPE if and only if $n = 1$, or $n$ is a prime
power, or $n= 2p$ for an odd prime $p$.
\end{mainthm}

\begin{mainthm}{Proposition \ref{thm:nilpotent-dominating-obstruction}}
    Let $G$ be finite, nilpotent, noncyclic, and not a $p$-group. If $\mathscr G_e(G)$ has a
nonidentity dominating vertex, then $G$ does not have the UEPE.
\end{mainthm}

\begin{mainthm}{Proposition \ref{thm:general-cyclicizer-obstruction}}
    Let \(G\) be a finite noncyclic group which is not a \(p\)-group.  If
\(\operatorname{Cyc}(G)\ne1\), then \(G\) does not have the UEPE.
\end{mainthm}

We thus obtain a classification for groups with nontrivial cyclicizer: 

\begin{mainthm}{Corollary \ref{cor:cyclicizer-exact}}
    Let \(G\) be a finite group with \(\operatorname{Cyc}(G)\ne1\).  Then \(G\)
has the UEPE if and only if $G$ is a $p$-group, or
 $G\cong C_{2p}$ for an odd prime $p$.
\end{mainthm}

Finally, we consider the proper power graph $\mathscr{G}^{*}(G)$. Removing
the identity removes the dominating vertex that the previous arguments rely
on, and $\mathscr{G}^{*}(G)$ may be disconnected. We therefore restrict to
the groups for which the solved case of the Overfull Conjecture still
applies, namely those for which $\mathscr{G}^{*}(G)$ has a vertex of degree
at least $|V(\mathscr{G}^{*}(G))|-3=|G|-4$. These cases are handled in Lemma \ref{lem:families} and Theorem \ref{thm:proper}. We also find a result by studying the core of $\mathscr G^*(G)$ (following Theorem \ref{fournier} of Fournier) in Proposition \ref{prop:forest} and solve the dihedral case: 

\begin{mainthm}{Proposition \ref{prop:dihedral}}
The proper power graph of a dihedral group of order $2^m$ $(m \geq 3)$ is Class 2.
\end{mainthm}


\section{Edge-coloring the power graph}\label{sec:power}

We will now consider the graph-theoretic concepts outlined above specifically for power graphs of finite groups. 

\begin{proposition}\label{even is not overfull}
    If $|G|$ is even, then $\mathscr{G}(G)$ is not overfull. 
\end{proposition}

\begin{proof}
    Say $|G| = 2n$. Then 
    \begin{align*}     \frac{|E(\mathscr{G}(G))|}{\lfloor 2n/2 \rfloor} > 2n-1 &\Leftrightarrow |E(\mathscr{G}(G))| > \frac{2n(2n-1)}{2}, \end{align*} which never holds since a complete graph on $2n$ vertices has $\binom{2n}{2} = \frac{2n(2n-1)}{2}$ edges. 
\end{proof}

Remark, as mentioned in the introduction, that this holds more generally for every graph of even order. 

\begin{lemma}\label{afford}
    If $G$ is a group of odd order with $|G| = 2n + 1$, then, to be overfull, $\mathscr{G}(G)$ can have at most $n-1$ edges removed compared to $K_{2n+1}$. 
\end{lemma}

\begin{proof}
    For $\mathscr{G}(G)$ to be overfull, we need 
    \begin{align*}
    \frac{|E(\mathscr{G}(G))|}{\lfloor (2n+1) /2 \rfloor} > 2n \Leftrightarrow \frac{|E(\mathscr{G}(G))|}{\lfloor 2n /2 \rfloor} > 2n \Leftrightarrow |E(\mathscr{G}(G))|> 2n^2. 
    \end{align*}
    Now, 
    \begin{align*}
        \frac{2n(2n+1)}{2} - (2n^2+1) = \frac{2n(2n+1-2n)}{2} -1 = n-1, 
    \end{align*}
    so we may lose up to $n-1$ edges if we want to remain overfull. 
\end{proof}

\begin{proposition}\label{only id join so not overfull}
    If only the identity is joined to all other vertices, then $\mathscr{G}(G)$ is not overfull. 
\end{proposition}

\begin{proof}
    From before, if $|G|$ is even, then $\mathscr{G}(G)$ is not overfull. So say $|G| = 2n+1$ and only the identity is joined to all other vertices. Then every other element, $2n$ of them, must have at least one vertex to which it is not adjacent. That's at least $n$ missing edges, hence breaking the requirement of the above lemma. Therefore if only the identity is joined to all other vertices, $\mathscr{G}(G)$ is not overfull.
\end{proof}

\bigskip

\begin{proposition}\label{odd cyclic is overfull}
     If $G$ is a cyclic group of odd prime-power order, then $\mathscr{G}(G)$ is overfull. 
\end{proposition}

\begin{proof}
    If $G$ is a cyclic group of odd prime power order, then $\mathscr{G}(G)$ is complete as shown in Proposition \ref{non trivial center implications}. In which case we have $\frac{2n(2n+1)}{2} = 2n^2+n$ edges which is more than the necessary $2n^2$.
\end{proof}

\begin{lemma}\label{lemma: cyclic not prime not ovf}
    If $G$ is a cyclic group of odd order which is not a prime power, then $\mathscr{G}(G)$ is not overfull. 
\end{lemma}

\begin{proof}
    Let $G$ be a cyclic group of odd order, $|G| = 2n+1$, and suppose $|G|$ is not a prime power. Since $|G|$ has at least two prime divisors, write $|G| = p^rq^sc$, where $p, q$ are distinct odd primes, $r, s, c \in \mathbb{N}$, and $\gcd(pq, c)=1$. 

    Now consider 
    \[A = H_{p^rc} \setminus H_c \text{ and } B = H_{q^sc} \setminus H_c\]
    where $H_{p^rc}$, $H_{q^sc}$, and $H_c$ are the unique subgroups of order $p^rc$, $q^sc$, and $c$ respectively. 
    It follows that $|A| = p^rc-c = (p^r - 1)c$ and $|B| = (q^s-1)c$. 

    Now in $\mathscr{G}(G)$, no element of $A$ and $B$ are adjacent since every element of $A$ has order divisible by $p$ but not $q$ and every element of $B$ has order divisible by $q$ but not $p$. Thus $\mathscr{G}(G)$ is missing at least $|A||B| = (p^r-1)(q^s-1)c^2$ edges. 

    Since $|G| = 2n+1 = p^rq^sc$, we have $n =  \frac{p^rq^sc-1}{2}$. So in order to not be overfull, $\mathscr{G}(G)$ must be missing at least $\frac{p^rq^sc-3}{2}$ edges. Since $p^r \geq 3$, $q^s \geq 5$, and $c\geq 1$, we have:
    \begin{align*}
        2(p^r-1)(q^s-1)c^2 - (p^rq^sc-3) &\geq 2(p^r-1)(q^s-1)c-p^rq^sc+3 \\
        & = c(2(p^r-1)(q^s-1)-p^rq^s)+3 \\
        & = c((p^r-2)(q^s-2)-2)+3 \\
        & \geq c+3 >0 \qquad \qquad \text{since } p^r-2 \geq 1 \text{ and } q^s-2 \geq 3
    \end{align*}
    Thus \[(p^r-1)(q^s-1)c^2 > \frac{p^rq^sc-3}{2} = n-1.\] 
    Hence $\mathscr{G}(G)$ is missing more than $n-1$ edges and, by Lemma \ref{afford}, is not overfull. 
\end{proof}

\bigskip

\begin{theorem}\label{ovf iff}
    Let $G$ be a finite group, then $\mathscr{G}(G)$ is overfull if and only if $G$ is cyclic of odd prime-power order. 
\end{theorem}

\begin{proof}
    For the right to left implication, Proposition \ref{odd cyclic is overfull} gives the right to left implication. 
    For the other direction, note that, as shown in Proposition \ref{even is not overfull}, if $|G|$ is even, then $\mathscr{G}(G)$ cannot be overfull. Also note that if only the identity is adjacent to all other vertices, then $\mathscr{G}(G)$ is not overfull (Proposition \ref{only id join so not overfull}). It follows from Proposition \ref{non trivial center implications} that, to be overfull, $G$ must be either a cyclic group, or the generalized quaternion. 

    Now the generalized quaternion group has even order so it cannot be overfull. So $G$ must be cyclic. Yet, as shown in Lemma \ref{lemma: cyclic not prime not ovf}, if $G$ is a cyclic group of odd order which is not a prime power, then $G$ is not overfull. Thus $G$ must be a cyclic group of odd prime power order and this completes our proof. 
\end{proof}

\bigskip

\begin{theorem}\label{thm:pg-class2}
    Let $G$ be a finite group, then $\mathscr{G}(G)$ is Class 2 if and only if it is a cyclic group of odd prime-power order. 
\end{theorem}

\begin{proof}
    For a group, $\Delta(\mathscr{G}(G))$ is always equal to $|V(\mathscr{G}(G))| - 1$ since at least the identity is adjacent to all others. Therefore, by Theorem \ref{overfull proven part}, $\mathscr{G}(G)$ is Class $2$ if and only if $\mathscr{G}(G)$ contains an overfull subgraph $H$ with $\Delta(H) = \Delta(\mathscr{G}(G))$. So it suffices to show that only $\mathscr{G}(C_{p^\alpha})$ contains an overfull subgraph. 

    Now, if $\mathscr{G}(G)$ had an overfull subgraph say $H$, since $\Delta(\mathscr{G}(G))$ is the highest possible degree attainable in any graph of that order, i.e., $n-1$, then no proper subgraph can match it since they wouldn't have enough vertices. So the overfull subgraph would have to be $\mathscr{G}(G)$ itself or a spanning subgraph. We showed in Theorem \ref{ovf iff}, that $\mathscr{G}(G)$ is overfull if and only if $G$ is cyclic of odd prime power order. If $H  = \mathscr{G}(G)$ we are done, if $H$ is a spanning subgraph then if $\mathscr{G}(G)$ is overfull we can disregard $H$, and if it is not then nor is $H$ by definition. 
    Thus $\mathscr{G}(G)$ is Class 2 if and only if $G$ is cyclic of odd prime power order.
\end{proof}

We remark that, as explained in the proof, for a finite group $G$, $\mathscr{G}(G)$ is Class 2 if and only if it is overfull. We also give an example of a $14$-coloring of $\mathscr{G}(C_{15})$ in the appendix. 

\section{The enhanced power graph and intersection power graph}

Remember that the enhanced power graph, $\mathscr{G}_{\mathrm{e}}(G)$, has vertex set $G$, and two distinct elements are adjacent if
they belong to a common cyclic subgroup. 

\begin{theorem}
\label{thm:enhanced-power}
Let $G$ be a finite group of order $n\geq 2$. Then
$\mathscr{G}_{\mathrm{e}}(G)$ is Class 2 if
and only if $G$ is cyclic of odd order.
\end{theorem}

\begin{proof}
The identity is adjacent to every other vertex of the enhanced power graph, so $\Delta(\mathscr{G}_{\mathrm{e}}(G))=n-1$. 
If $n$ is even, the graph is of Class 1, so suppose that $n = 2k+1$ is odd. If $G$ is cyclic, then $\mathscr{G}_{\mathrm{e}}(G)\cong K_n$, which is overfull and hence Class 2.

Now suppose that $G$ is noncyclic. Let $M=\langle x\rangle$ be a maximal
cyclic subgroup and put $m=|M|$. It follows that $x$ has exactly $ 2k+1 -m$ non-neighbours. Since $M$ is a proper subgroup and $n = 2k+1$ its index is an odd integer at
least $3$. Therefore $m\leq (2k+1)/3$, and the enhanced power graph is missing at
least $2k+1-m \geq 2(2k+1)/3\geq k$ edges so it is not overfull by Lemma \ref{afford} and is hence Class 1. 
\end{proof}

\bigskip

The intersection power graph of $G$, introduced by Bera \cite{Bera2018} and
denoted by $\mathscr{G}_{\mathrm{I}}(G)$, has vertex set $G$. The identity is
adjacent to every other vertex, and two distinct nonidentity elements $x,y$
are adjacent if and only if
$\langle x\rangle\cap\langle y\rangle\neq\{e\}$.

\begin{theorem}
\label{thm:intersection-power}
Let $G$ be a finite group of order $n\geq 2$. Then
$\mathscr{G}_{\mathrm{I}}(G)$ is Class 2 if and only if  $G$ is cyclic of odd prime-power order. 
\end{theorem}

\begin{proof}
The identity element is adjacent to every other vertex of the
intersection power graph, so
$\Delta(\mathscr{G}_{\mathrm{I}}(G))=n-1$. Also note that if $G$ has even order, then $\mathscr{G}_{\mathrm{I}}(G)$ is Class 1. 

Suppose from now on that $n$ is
odd. If $G\cong C_{p^a}$, then any two nontrivial subgroups of $G$ have
nontrivial intersection. Therefore
$\mathscr{G}_{\mathrm{I}}(G)\cong K_n$, and its chromatic index is $n$.

Suppose next that $G$ is a noncyclic $p$-group. Since $p$ is odd, a finite
$p$-group with a unique subgroup of order $p$ is cyclic. Thus $G$ contains
two distinct subgroups $P=\langle x\rangle$ and $Q=\langle y\rangle$ of
order $p$. The vertices $x$ and $y$ are nonadjacent, and no nonidentity
element $z$ can be adjacent to both. Indeed, this would imply
$P,Q\leq\langle z\rangle$, which is impossible because a cyclic $p$-group
has a unique subgroup of order $p$.

For every $z\in G\setminus\{e,x,y\}$, at least one of $xz$ and $yz$ is
therefore a missing edge. Together with the missing edge $xy$, this gives at
least $n-2\geq(n-1)/2$ missing edges. Thus the graph is not overfull and is
of Class 1.

Finally, suppose that at least two distinct primes divide $n$. By the
Feit--Thompson theorem, $G$ is solvable. Choose a
nonempty proper subset $\pi$ of the prime divisors of $n$, and write $n=ab$,
where $a$ is the full $\pi$-part of $n$ and $b$ is the full $\pi'$-part.
Then $a,b>1$ and $\gcd(a,b)=1$. Hall's theorem gives
subgroups $H,K\leq G$ with $|H|=a$ and $|K|=b$.

Since $H\cap K=\{e\}$, no element of $H\setminus\{e\}$ is adjacent to an
element of $K\setminus\{e\}$. The graph is consequently missing at least
$(a-1)(b-1)$ edges. Since $a$ and $b$ are coprime odd integers greater than
$1$, one of them is at least $5$, and
\[
 (a-1)(b-1)-\frac{n-1}{2}
   =\frac{(a-2)(b-2)-1}{2}\geq 0.
\]
The graph is therefore not overfull and is Class 1. This completes the proof. 
\end{proof}

Thus, although determining whether a general graph is Class 1 or Class 2 is NP-complete, for graphs defined on groups the distinction can admit a complete algebraic characterization: the exceptional Class 2 graphs arise from specific group-theoretic family.

\section{Extending colorings}\label{sec:uepe}

Since the power graph is a spanning subgraph of both the enhanced power graph and the
intersection power graph, it is natural to ask not only how their chromatic indices compare,
but also whether an edge-coloring of the power graph can be preserved when the additional
edges are added. This is an instance of the general edge-precoloring extension problem, in
which one asks whether a prescribed coloring of a subgraph extends to the whole graph
without recoloring the precolored edges; see, for example, \cite{GiraoKang2019}. Here we
consider the stronger requirement that every optimal edge-coloring of the power graph admit
such an extension to the enhanced power graph.

Existence is trivial since, provided we fall in the case where the graphs have the same edge chromatic number, any edge coloring of the enhanced power graph (resp. intersection power graph) is an edge coloring of the power graph after removal of some edges. Hence that specific coloring of the power graph extends to the enhanced power graph (resp. intersection power graph). We thus investigate for which groups this is always the case. 

Henceforth, we say that a graph $\Gamma$ has the \emph{universal edge-precoloring extension} (UEPE) to graph $\Theta$ if every optimal edge-coloring of $\Gamma$ extends, with the same color set, to graph $\Theta$, without recoloring the edges of $\Gamma$. In our case, we will consider this extension from the power graph to the enhanced power graph.

In \cite{BiswasCameronDasDey2024}, the authors introduce the \textit{difference graph} $D(G) = \mathscr{G}_e(G) - \mathscr{G}(G)$ as the difference of the enhanced power graph and power graph of a group, with all isolated vertices removed. This will be useful to us as the edges of the difference graph are exactly the edges we need to color to show that a group has UEPE.  

\medskip

Going forward, denote by $d_\Gamma(x)$ the degree of $x$ in $\Gamma$ and by $N_{\Gamma}(x)$ the open neighborhood of $x$ in $\Gamma$. If $\varphi$ is an edge-coloring of $\mathscr G(G)$ with color set $C$, then for $x\in G$ we define $M_\varphi(x)=C\setminus\{\text{colors already used on edges of }\mathscr G(G)\text{ incident with }x\}$. We begin with an identification tool. 

\begin{theorem}\label{thm:criterion}
Let \(G\) be a finite group of order \(n\) such that \(\mathscr G(G)\) is Class \(1\). If $d_{\mathscr G_e(G)}(x)+d_{\mathscr G_e(G)}(y)\le n$ for every \(xy\in E(D(G))\), then \(G\) has the UEPE.

\end{theorem}


\begin{proof}
Let \(\varphi\) be an optimal edge-coloring of \(\mathscr G(G)\). Since \(\mathscr G(G)\) is Class \(1\), \(\varphi\) uses \(n-1\) colors. A new edge \(xy\in E(D(G))\) may receive exactly the colors in \(M_\varphi(x)\cap M_\varphi(y)\). Since $|M_\varphi(x)|=n-1-d_{\mathscr G(G)}(x)$,
we have
\[
|M_\varphi(x)\cap M_\varphi(y)|
\ge n-1-d_{\mathscr G(G)}(x)-d_{\mathscr G(G)}(y).
\]
Using \(d_{\mathscr G(G)}(x)=d_{\mathscr G_e(G)}(x)-d_{D(G)}(x)\) and the hypothesis that $d_{\mathscr G_e(G)}(x) + d_{\mathscr G_e(G)}(y) \leq n$, this gives
\[
|M_\varphi(x)\cap M_\varphi(y)|
\ge d_{D(G)}(x)+d_{D(G)}(y)-1.
\]
Since \(xy\) is adjacent in \(D(G)\) to at most \(d_{D(G)}(x)+d_{D(G)}(y)-2\) other edges, the new edges can be colored.  
\end{proof}

An immediate corollary is worth noting. 

\begin{corollary}\label{cor:centerless-dihedral}
Every finite centerless group and every finite dihedral group have the UEPE.
\end{corollary}

\begin{proof}
Suppose first that \(Z(G)=1\). If \(x\neq1\), then \(N_{\mathscr G_e(G)}[x]\subseteq C_G(x)\). Since \(C_G(x)\) is proper, \(|C_G(x)|\le |G|/2\), and hence \(d_{\mathscr G_e(G)}(x)\le |G|/2-1\). The identity is incident with no edge of \(D(G)\), so every \(xy\in E(D(G))\) satisfies $d_{\mathscr G_e(G)}(x)+d_{\mathscr G_e(G)}(y)\le |G|-2$.
Theorem~\ref{thm:criterion} applies.

Now let \(G=D_{2m}\). Every new edge has both endpoints in the cyclic rotation subgroup, and each such endpoint has enhanced degree \(m-1\). Hence every \(xy\in E(D(G))\) satisfies $d_{\mathscr G_e(G)}(x)+d_{\mathscr G_e(G)}(y)=2m-2<|G|$,
so Theorem~\ref{thm:criterion} again applies.
\end{proof}

\bigskip

Turning to cyclic groups we can get an exact classification. 

\begin{theorem}\label{thm:cyclic-exact}
The cyclic group $C_n$ has the UEPE if and only
if one of the following conditions holds: 
\begin{itemize}
    \item $n=1$; 
    \item $n$ is a prime power; 
    \item $n=2p$ for an odd prime $p$. 
\end{itemize}
\end{theorem}

\begin{proof}
The case $n=1$ is vacuous. If $n$ is a prime power, then $\mathscr G(G) = K_n = \mathscr G_e(G)$ so the result is immediately true. 

If $n$ is odd and is not a prime power, the two graphs have different edge chromatic number as seen before so no optimal coloring can extend. 

Now let $n=2p$, where $p$ is an odd prime, and let $t$ be the involution.
The edges that are new in $\mathscr G_e(G)$ are precisely the edges $tx$, where $x$ runs through the
$p-1$ nonidentity elements of the subgroup of order $p$. In an optimal
$(2p-1)$-edge-coloring, $t$ misses $p-1$ colors and each such $x$ misses
one color. Every color class is a matching on an even number of vertices,
so every color is missing at an even number of vertices. The identity and
the generators miss no color. So for any color $\alpha$ missing at $t$, it is missing at one non-generator, and vice versa. Consequently the missing-color incidences
at the $p-1$ vertices $x$ give a bijection with the colors missing at $t$.
Coloring $tx$ with the unique color missing at $x$ gives the extension.

It remains to show that every other even non-prime-power $n$ does not have the UEPE.  Write $n=2^a m$, $m>1$ odd.

Unless $a=1$ and $m$ is prime, $D(C_n)$ contains two vertex-disjoint edges.
Indeed, take inverse pairs of elements $d_1, d_2$ having the following incomparable orders: take $d_1=4$ and $d_2=r$ for a prime $r\mid m$ if
$a\ge2$; take two distinct prime divisors of $m$ if $a=1$ and $m$
is not a prime power; and take $d_1=2p^{k-1}$ and $d_2=p^k$ if
$m=p^k$ with $k\ge2$. Choose inverse pairs
$\{y,y^{-1}\}$ and $\{z,z^{-1}\}$ of elements of orders $d_1$ and
$d_2$. Since these orders are incomparable, $yz$ and
$y^{-1}z^{-1}$ are two vertex-disjoint edges of $D(C_n)$.

Now choose an element $x$ of odd prime order, and let $t$ be the involution.
Let
\[
 M=\bigl\{\{1,t\}\bigr\}
   \cup
   \bigl\{\{y,y^{-1}\}:y\notin\{1,t,x,x^{-1}\}\bigr\},
\]
where each inverse pair is taken once.  This is a matching of size $1 + (n-4)/2 = n/2-1$
in $\mathscr G(C_n)$, and its only uncovered vertices are $x,x^{-1}$.
Put $H=\mathscr G(C_n)-M$.  All dominating vertices of
$\mathscr G(C_n)$ are covered by $M$, while $x$ and $x^{-1}$ are not
dominating, so $\Delta(H)=n-2$.

Now $H$ is Class 1. Indeed, it falls into the solved part of the overfull conjecture, so an $(n-2)$overfull subgraph would need $n-1$ vertices and be some $H-v$. But regardless of whether $v$ is covered by $M$, $H-v$ would be missing $n/2-1$ edges from $K_{n-1}$ and is hence not overfull by Lemma \ref{afford} (minding notation). 

Per our previous coloring results, we get an $(n-2)$-edge-coloring of $H$.
Give every edge of $M$ one new color $\alpha$.  This is an optimal
$(n-1)$-edge-coloring of $\mathscr G(C_n)$.  In an extension to $K_n$, the
$\alpha$-class would have to become a perfect matching; its only possible
additional edge is $xx^{-1}$.  But $xx^{-1}$ already belongs to
$\mathscr G(C_n)$ and has a fixed color different from $\alpha$. Thus this
coloring does not extend and this completes the proof. 
\end{proof}

Note that obstruction to having UEPE of Theorem \ref{thm:cyclic-exact} can be used again and works as follows. Suppose that
$z$ is a dominating vertex of $\mathscr G_e(G)$ and that
$M\subseteq E(\mathscr G(G))$ is a matching which avoids $z$ and
covers every vertex of $N_{D(G)}(z)$. If
$\mathscr G(G)-M$ admits an $(|G|-2)$-edge-coloring, then coloring
all edges of $M$ with one new color $\alpha$ gives an optimal
edge-coloring of $\mathscr G(G)$ which cannot extend to
$\mathscr G_e(G)$. Indeed, $\alpha$ is missing at $z$ but occurs at
every vertex of $N_{D(G)}(z)$. Since $z$ is dominating in
$\mathscr G_e(G)$, any extension would have to use $\alpha$ on a
new edge incident with $z$, which is impossible. 

\medskip

Actually, we find that having a nonidentity dominating vertex for $\mathscr{G}_e(G)$ can be problematic for some groups. Before that, note that Bera and Dey showed in \cite{BeraDey2022} the following. Let $G$ be a finite nilpotent group $G=G_0\times C\times Q$,
where $C$ is the product of the cyclic Sylow subgroups of $G$, $Q$ is the
generalized-quaternion Sylow $2$-subgroup if one exists (and $Q=1$
otherwise), and $G_0$ is the product of the remaining Sylow subgroups.
Then $\operatorname{Dom}(\mathscr G_e(G))
=
\{1_{G_0}\}\times C\times T$,
where $T=\{1\}$ if $Q=1$, while $T=\{1,z\}$ if $Q$ is generalized
quaternion, with $z$ its unique involution.

\begin{proposition}\label{thm:nilpotent-dominating-obstruction}
Let $G$ be finite, nilpotent, noncyclic, and not a $p$-group.  If $\mathscr G_e(G)$ has a nonidentity dominating vertex, then $G$ does not have the UEPE.
\end{proposition}

\begin{proof}
We can use the obstruction described above. Choose a Sylow factor $S$ as follows:
if some Sylow $p$-subgroup is cyclic, take such an $S$ and let $z$ generate
its unique subgroup of order $p$; otherwise take the generalized-quaternion
Sylow $2$-subgroup $S$ and let $z$ be its unique involution. Write $G=S\times H$, $|S|=s$, $|H|=m$, $|G|=n=sm$.
Choose $S$ cyclic so that $H$ is noncyclic, and so
$m\geq 4$. 

Moreover, $N_{D(G)}((z,1))=\{(1,h):h\in H\setminus\{1\}\}$ since $(z, 1)$ and $(1, h)$ are never adjacent since they come from different Sylow factors and $\langle (z,1),(1,h)\rangle = \langle(z,h)\rangle$. 

Now fix $h_0\in H\setminus\{1\}$ and let
\[
M=
\bigl\{\{(1,1),(1,h_0)\}\bigr\}
\cup
\bigl\{\{(1,h),(z,h)\}:h\in H\setminus\{1,h_0\}\bigr\}.
\]
Since $(1,h)\in\langle(z,h)\rangle$, this is a matching in
$\mathscr G(G)$ which avoids $(z,1)$ and covers $N_{D(G)}((z,1))$.

Put $\Gamma=\mathscr G(G)-M$. Since the identity is the unique dominating
vertex of $\mathscr G(G)$, we have $\Delta(\Gamma)=n-2$. We claim that
$\Gamma$ is Class $1$. Again it suffices to rule out an
$(n-2)$-overfull subgraph.

There are directly $(s-1)(m-1)$ power graph nonedges of the form
$\{(a,1),(1,h)\}$ with $a\neq1$ and $h\neq1$, but there is more. If $n$ is odd, take two nonadjacent vertices of $H$ (since noncyclic), say $u,v$. The $s^2$ pairs between $S\times\{u\}$ and
$S\times\{v\}$ are also nonedges, and together with the $m-1$ edges of
$M$ this gives at least
\[
(s-1)(m-1)+s^2+(m-1)
=s(m-1)+s^2\geq n-1
\]
missing edges. Hence $\Gamma$ is not overfull.

If $n$ is even, only subgraphs on $n-1$ vertices need be considered.
For every $v\in G$, the number of edges missing from $\Gamma-v$ relative
to $K_{n-1}$ is at least
\[
\begin{cases}
(s-1)(m-1), & v\in(S\setminus\{1\})\times\{1\},\\
s(m-2), & v\in\{1\}\times(H\setminus\{1\}),\\
sm-s-1, & \text{otherwise}.
\end{cases}
\]
Each quantity is at least $(n-2)/2$, since $s\geq2$ and $m\geq4$.
Thus no $\Gamma-v$ is overfull.

Therefore $\chi'(\Gamma)=n-2$. The preceding obstruction now applies to
the matching $M$ and the dominating vertex $z$, so $\mathscr G(G)$ has
an optimal edge-coloring which does not extend to $\mathscr G_e(G)$ and $G$ does not have the UEPE. 
\end{proof}

\medskip

We can finish this section with a group theoretic constraint to UEPE. Remember that $\operatorname{Cyc}(G)=\{z\in G:\langle z,g\rangle\text{ is cyclic for every }g\in G\}$. 

\begin{proposition}
\label{thm:general-cyclicizer-obstruction}
Let \(G\) be a finite noncyclic group which is not a \(p\)-group.  If
\(\operatorname{Cyc}(G)\ne1\), then \(G\) does not have the UEPE.
\end{proposition}

\begin{proof}
The nilpotent case is Proposition~\ref{thm:nilpotent-dominating-obstruction}, so
suppose that $G$ is not nilpotent. In particular, $G$ is nonabelian. We rely again on the same obstruction as before. 

Choose $z\in\operatorname{Cyc}(G)$ of prime order $p$, and put
$B=N_{D(G)}(z)$. Since $z$ is dominating in $\mathscr G_e(G)$, an element
$g\neq1$ belongs to $B$ exactly when $p\nmid o(g)$. Indeed,
$\langle z,g\rangle$ is cyclic; if $p\mid o(g)$, then uniqueness of the
subgroup of order $p$ gives $z\in\langle g\rangle$, whereas if
$p\nmid o(g)$, the cyclic subgroups $\langle z\rangle$ and $\langle g\rangle$
are incomparable. If $p$ is odd, take
\[
M=\bigl\{\{g,zg\}:g\in B\bigr\}\cup\bigl\{\{1,z^2\}\bigr\}.
\]
If $p=2$, pair the elements of $B$ into inverse pairs, choose $g_0\in B$,
and add the edge $\{1,zg_0\}$. In either case, $M$ is a matching in
$\mathscr G(G)$ which avoids $z$, covers $B=N_{D(G)}(z)$, and contains
an edge incident with the identity.

Put $\Gamma=\mathscr G(G)-M$. Since the identity is the unique dominating
vertex of the power graph of a noncyclic non-$p$-group,
$\Delta(\Gamma)=n-2$, where $n=|G|$.

We see that $\Gamma$ is Class~$1$. By Gustafson's bound in \cite{Gustafson1973},
a nonabelian group has at least $3n^2/16$ unordered pairs of
noncommuting elements, each of which is a nonedge of $\mathscr G(G)$.
If $n$ is odd, only $\Gamma$ itself can be $(n-2)$-overfull, but
$3n^2/16\ge n-1$, so this is impossible.

If $n$ is even, only subgraphs on $n-1$ vertices need be considered.
After deleting any vertex, at least
\[
\frac{3n^2}{16}-(n-1)\ge\frac{n-2}{2}
\]
noncommuting pairs remain for $n\ge8$, so no such subgraph is overfull.
The only nonabelian non-$p$-group of smaller order is $S_3$, whose
cyclicizer is trivial.

Hence $\Gamma$ has no $(n-2)$-overfull subgraph and $\chi'(\Gamma)=n-2$. The preceding obstruction now applies to $M$ and
$z$, so $G$ does not have the UEPE. 
\end{proof}

Together with Theorem~\ref{thm:cyclic-exact}, this gives a
complete classification whenever the cyclicizer is nontrivial.

\begin{corollary}
\label{cor:cyclicizer-exact}
Let \(G\) be a finite group with \(\operatorname{Cyc}(G)\ne1\).  Then \(G\)
has the UEPE if and only if $G$ is a $p$-group, or
 $G\cong C_{2p}$ for an odd prime p.
\end{corollary}

\section{The proper power graph}\label{sec:proper}

We treat the proper power graph separately since, as opposed to other graphs studied this far, we remove a vertex: the identity. We lose the solved cases of the overfull conjecture as a tool. We also find ourselves in the ``opposite" context of our previous results since, if $G$ has odd order, then $\mathscr{G}^*(G)$ has an even number of vertices. Moreover, since $\mathscr{G}^*(G)$ need not have a vertex of degree high enough to satisfy the solved part of the overfull conjecture, we cannot directly dismiss any groups based on parity of the order.

Since the proper power graph doesn't have a vertex for the identity of $G$, $\mathscr{G}^*(G)$ may be disconnected, e.g. $S_3$, $C_2 \times C_2$, and more generally $C_p \times C_p$ since it has one component for each subgroup of order $p$. As such, determining whether the proper power graph is Class 1 or Class 2 becomes a question of studying the connected components which contain the vertices of maximum degree. We study below some of the more straightforward cases for which our tools apply. 

\bigskip

Recall the following theorem from \cite{C-G-S}: 

\begin{theorem}
    A finite group has a complete undirected power graph if and only if it is cyclic and has order equal to $p^m$, where $p$ is a prime and $m$ is a nonnegative integer. 
\end{theorem}

The next proposition follows. 

\begin{proposition}\label{prop:c2m}
    $\mathscr{G}^*(C_{2^m})$ with $m\geq 2$ is Class 2. 
\end{proposition}

\begin{proof}
    $\mathscr{G}(C_{2^m})$ is complete of even order. Thus $\mathscr{G}(C_{2^m}) \setminus \{e\} = \mathscr{G}^*(C_{2^m})$ is complete of odd order and hence Class 2. 
\end{proof}

\begin{proposition}\label{prop:dihedral}
    The proper power graph of a dihedral group of order $2^m$ $(m\geq 3)$ is Class 2. 
\end{proposition}

\begin{proof} 
Write $G=\langle r,s\mid r^{2^{m-1}}=s^2=1,\ srs=r^{-1}\rangle$. For $m=1$ and $m=2$, we get $C_2$ and $C_2 \times C_2$ which are Class 1. Then for $m \geq 3$, $\langle r\rangle\cong C_{2^{m-1}}$ is the cyclic subgroup of rotations, while the remaining $2^{m-1}$ elements are reflections, each of order $2$, hence isolated in $\mathscr G^*(G)$. Thus $\mathscr G^*(G)=\mathscr G^*(\langle r\rangle)\dot\cup (2^{m-1})K_1$.

Since $\langle r\rangle$ is cyclic of order $2^{m-1}$, its subgroups are linearly ordered by inclusion. Thus any two nonidentity elements of $\langle r\rangle$ are powers of one another, and so $\mathscr G^*(\langle r\rangle)\cong K_{2^{m-1}-1}$. Therefore $\mathscr G^*(G)\cong K_{2^{m-1}-1}\dot\cup (2^{m-1})K_1$.
Now $2^{m-1}-1$ is odd, $\mathscr G^*(G)$ is Class $2$.
\end{proof}

In \cite{MoghaddamfarRahbariyanShi}, the authors show that a non trivial finite group has bipartite proper power graph if and only if $\pi_e(G) \subseteq \{1, 2, 3\}$. It follows that groups satisfying this condition have Class 1 proper power graph. 

Using Fournier's result that if the core of the graph contains no cycles, then the graph is Class 1, can also give us results.

\begin{proposition} \label{prop:forest}
The core of the proper power graph of a group $G$, denoted $\mathscr{G}^{*}(G)_{\Delta}$, is a forest if and only if,
among the cyclic subgroups generated by maximum-degree elements, no one properly contains another and each has order $2$, $3$, or $6$.

In this case, $\mathscr{G}^{*}(G)_{\Delta}$ is a disjoint union of isolated vertices and independent edges.
\end{proposition}

\begin{proof}
For a nontrivial cyclic subgroup $H\leq G$, define
\[
w_G(H)=|H|-2+
\sum_{\substack{H<K\leq G\\ K\text{ cyclic}}}\varphi(|K|).
\]
If $x$ is a generator of $H$, then $\deg(x)=w_G(H)$. Let $\Delta$ be the maximum of these values and let $\mathcal D_G$ be the family of cyclic subgroups
$H$ with $w_G(H)=\Delta$. The maximum-degree vertices are exactly
the generators of the members of $\mathcal D_G$.

For each $H\in\mathcal D_G$, its generators induce a complete graph on
$\varphi(|H|)$ vertices. Also remark that, if $H,K\in\mathcal D_G$ and $H<K$, then the generators of $H$ and $K$ together induce a clique of order at least $3$.
Hence, for $\mathscr{G}^{*}(G)_{\Delta}$ to be a forest, $\mathcal D_G$ has to be an antichain and $\varphi(|H|)\leq2$ for every $H\in\mathcal D_G$; so $|H|\in\{2,3,4,6\}$.

The case $|H|=4$ cannot occur: if $L<H$ has order $2$, then the contribution
$\varphi(4)=2$ of $H$ to $w_G(L)$ gives $w_G(L)\geq w_G(H)=\Delta$.
Thus $L\in\mathcal D_G$, contradicting the antichain property.

Conversely, suppose that $\mathcal D_G$ is an antichain and every
$H\in\mathcal D_G$ has order $2$, $3$, or $6$. Distinct generator sets
then have no edges between them, while each induces either $K_1$ or $K_2$.
Hence $\mathscr{G}^{*}(G)_{\Delta}$ is a union of isolated vertices and edges, so a forest. 
\end{proof}

It follows that for these groups $\mathscr G^* (G)$ is Class 1. 

\medskip

Another natural option to study the edge coloring of proper power graphs is to find the cases where our previous tools apply, i.e. when $\Delta(\Gamma) \geq |V(\Gamma)| - 3$ and we can use the solved part of the overfull conjecture. We give some families below and conjecture that these are the exact families.

\begin{lemma}\label{lem:families}
Let $G$ be a finite nontrivial group and put $v=|G|-1$. Then $\mathscr G^*(G)$ contains a
vertex $x$ such that $\deg_{\mathscr G^*(G)}(x)\geq v-3$
if $G$ belongs to one of the following families:
\begin{enumerate}[i)]
    \item $G$ is cyclic;
    \item $G$ is a generalized quaternion $2$-group;
    \item $G$ is a noncyclic group of the form $C_3\rtimes_\omega P$,
    where $P$ is a cyclic or generalized quaternion $2$-group and the unique
    involution $z$ of $P$ belongs to $\ker(\alpha)$;
    \item $G$ is a $2$-group with exactly three involutions and there is an
    involution $z$ contained in every cyclic subgroup of order at least $4$.
\end{enumerate}
\end{lemma}

\begin{proof}
If $G$ is cyclic, a generator is adjacent to every other nonidentity element.
If $G$ is generalized quaternion, its unique involution is adjacent to every other vertex.

Now let $G=C_3\rtimes_\alpha P$ be as in \textit{(iii)}, and let $z$ be the unique involution of $P$.
Since $z\in\ker(\alpha)$, it is also the unique involution of $G$. The only nonidentity odd-order elements are the two nonidentity elements of
$C_3$, and hence these are precisely the two nonneighbors of $z$.

Finally, in case \textit{(iv)}, every non-involution generates a cyclic subgroup of order at least
$4$ containing $z$. Hence the only nonneighbors of $z$ are the other two involutions.
Thus in every case there is a vertex $x$ with at most two nonneighbors, equivalently
$\deg_{\mathscr{G}^*(G)}(x)\geq v-3$.

\end{proof}

We can find the edge coloring number of these groups: 

\begin{theorem}\label{thm:proper}
The proper power graph of the following groups is Class \(1\):
\begin{enumerate}[i)]
    \item \(G\) is cyclic and $G \not\cong C_{2^{k}}$ for every $k\geq 2$; 

    \item \(G\) is a generalized quaternion \(2\)-group;

    \item $G \cong C_{3}\rtimes_{\alpha}P$,
    where \(P\) is a nontrivial cyclic \(2\)-group or a generalized
    quaternion \(2\)-group, and the unique involution \(z\in P\)
    satisfies \(\alpha(z)=1\);

    \item \(G\) is a \(2\)-group with exactly three involutions, and
    there exists an involution \(z\in G\) which belongs to every cyclic
    subgroup of \(G\) of order at least \(4\).
\end{enumerate}
\end{theorem}

\begin{proof}
For \textit{(i)}, If \(n\) is odd, $\mathscr{G}^*(G)$ is even and hence Class 1. If $n$ is even and not a $2$-power, say \(n=2^a d\), \(d>1\) odd, the graph has at least
$(2^a-1)(d-1)\ge \frac{n-2}{2}$ missing edges, so it is not overfull and is Class 1. 

\noindent For \textit{(ii)}, in a generalized quaternion group, the unique involution is the sole maximum-degree vertex of \(\mathscr G^*(G)\). Thus, the core of the graph has only one vertex which makes it Class 1. 

\noindent For \textit{(iii)}, in \(C_3\rtimes_\alpha P\), the specified \(z\) is adjacent to every nonidentity element except the two elements of order \(3\), so
$d(z)=n-4$. In the noncyclic cases it is the unique maximum-degree vertex; the cyclic overlap has order divisible by \(6\) and is already Class 1 by the first case.

\noindent Finally in case \textit{(iv)}, write the involutions as \(z,u,v\). The vertices \(u,v\) are isolated, while \(z\) is adjacent to every noninvolution. For any noninvolution \(x\), \(ux\) supplies a third nonneighbor of \(x\); otherwise \(\langle u,x\rangle\) would produce more than three involutions. Hence \(z\) is again the unique maximum-degree vertex. Note that \(C_2\times C_2\) has an edgeless proper power graph and is Class 1.

\end{proof}

\bigskip

\pagebreak

\section*{Appendix}
\setcounter{theorem}{0}
\renewcommand{\thetheorem}{A.\arabic{theorem}}

\bigskip 

\begin{example}{\textbf{A coloring of $\mathscr{G}(C_{15})$}}

The following table gives a $14$-coloring of $\mathscr{G}(C_{15})$. The upper row represents the colors attributed to each edge, and each edge $(c^i, c^j)$ is simply denoted $(i, j)$. We used the coloring used by Behzad, Chartrand, and Cooper in \cite{coloring}: Label the vertices of $G$ $1, 2, \ldots, n$, and let
\[
S_p = \{ (p-q,\, p+q) \mid q = 1,\, 2,\, \ldots,\, (n-1)/2\}
\]
for $p = 1,\, 2,\, \ldots,\, n$, where for the edge $(p-q,\, p+q)$, each of the numbers $p-q$ and $p+q$ is expressed as one of the numbers $1,\, 2,\, \ldots,\, n$ modulo $n$. But we limited ourselves to $14$ colors and, once no more colors were available, laboriously manually moved around some edges to exchange colors when possible, and eventually free up a color for each edge that needed one.

\bigskip

\begin{table}[ht]
\caption{Coloring table of the $97$ edges of $\mathscr{G}(C_{15})$ in $14$ colors. }
{\small
\centering
\hspace*{-0.9cm} 
\begin{tabular}[t]{|C|C|C|C|C|C|C|}
\hline
 1& 2& 3& 4& 5& 6& 7 \\ \hline
 (15, 2)&(1, 3) & (2, 4) & (2, 6)& (4,6)& (5,7)& (6,8)\\ \hline
 (14, 3)&(15, 4)& (1, 5)& (1, 7)& (3,7)& (4,8)& (4,10) \\ \hline
 (13, 4)&(14, 5)& (15, 6)& (15,8)& (2,8)& (3,9)& (3,11) \\ \hline
 (11, 6)&(12, 7)& (14, 7)& (14,9)& (1,9)& (2,10)& (2, 12)\\ \hline
 (10, 7)&(11, 8)& (13, 8)& (13,10)& (15,10)& (1,11)& (1,13)\\ \hline
 (9, 8)&(2, 13)& (12, 9)& (3,12)& (14,11)& (15,12)& (15,14) \\ \hline
 (1, 12)&      & (11, 10)& (11,4)& (13,12)& (14,13)& (7,9) \\ \hline
\end{tabular} 

\hspace*{-0.9cm}
\begin{tabular}[t]{|C|C|C|C|C|C|C|}
\hline
 8&9 & 10& 11 & 12 & 13 & 14 \\ \hline
(5, 11)& (8,10) &(9,11) &(9,13) &(11,13) &(11,15) & (10,1) \\ \hline
(4,12)& (7,11)& (8,12)& (8,14)&(10,14) &(9,2) &(13,15) \\ \hline
(3,13)& (6,12)& (7,13)&(7,15) &(9, 15) &(8,3) &(11,2) \\ \hline
 (2,14)& (5,13)& (6,14)&(6,1) &(8,1) &(7,4) &(9,4) \\ \hline
 (1,15)& (4,14)& (5,15)&(5,2) &(7,2) &(5,10) &(8,5) \\ \hline
 (6,9)& (3,15)& (4,1)&(4,3) & (6,3)&(13,6) & (7,6) \\ \hline
(7,8)& (2,1)& (3,2)& (11,12)& (5,4)&(1,14) & (14,12) \\ \hline
\end{tabular} 

}
\end{table}

\end{example}

\bigskip

\begin{remark}

The proof of Theorem \ref{overfull proven part} for $\Delta(\Gamma) = |V(\Gamma)| - 1$, the case applicable to power graphs of finite groups, was done in \cite{Plantholt-1}. It is nevertheless not obvious to the author how it handles graphs which are not overfull but whose edges that were removed are not independent. The proof starts by considering the vertices of the complement of $\Gamma$, which have at least one edge incident to them, and constructing a bipartite graph with these as one of the parts. For $V(\Gamma) = 2n+1$ their bipartite graph is $K_{n+1, n}$. This is not possible for $\mathscr{G}(C_{15})$ since there are only $6$ vertices with edges incident to them in $\overline{\mathscr{G}(C_{15})}$: $c^5, c^{10}, c^3, c^6, c^9, c^{12}$. For this reason, we find it better to rely on the construction of \cite{Rhee}. Rhee's proof brings together three key points. He first shows that if two graphs have same cardinality edge sets, we can derive one from the other by relying on the edges of its complement and exchanging them with existing edges. He then shows that $K_{2n+1}-n$, the complete graph $K_{2n+1}$ deprived of $n$ edges, is $2n$ colorable if the edges removed are independent. He finally shows that for two graphs $K_{2n+1}-n$, if one is derivable from the other using the above technique, then they have the same edge-chromatic number. Bringing these together gives the desired result. 
\end{remark}
Here is a short example: 
Assume we have $\mathscr{G}(C_{15}) - c^{10}c^{5} + c^6 c^{10}$, with the same coloring as $\mathscr{G}(C_{15})$ in our previous coloring, but where the edge ${\{c^{10}, c^{5}\}}$ does not exist and where $\{c^6, c^{10}\}$ is colored by the color $2$. Following the construction of \cite{Rhee}, we observe that $\chi'(c^{10} c^{11}) = 2$ and that $c^5$ is not adjacent to any edge colored $7$. We follow the color-alternating path of colors $2$ and $7$ starting from $c^5$ to obtain the following vertex sequence: $c^5c^{14}c^{15}c^4c^{10}$ where $\chi'(c^5c^{14}) = 2$, $\chi'(c^{14}c^{15}) = 7$, $\chi'(c^{15}c^{4}) = 2$, $\chi'(c^4c^{10}) = 7$ and $c^{10}$ is not adjacent to an edge colored $2$. We can then invert the color of each edge in that path to make $\chi'(c^{10}c^5) = 2$ with no conflict.

\bigskip

\begin{example}{\textbf{A coloring of $\mathscr{G}(C_{15})$ using Rhee's construction}}

First, using the coloring of \cite{coloring} on $K_{15}$ up to 14 colors gives the following coloring table: 

\begin{table}[ht]
\caption{Coloring table of $K_{15}$ up to what $14$ colors allow. }
{\small
\centering
\hspace*{-0.9cm} 
\begin{tabular}[t]{|C|C|C|C|C|C|C|}
\hline
 1& 2& 3& 4& 5& 6&7  \\ \hline
 (15, 2)&(1, 3) & (2, 4) & (3,5)& (4,6)& (5,7)& (6,8)\\ \hline
 
 (14, 3)&(15, 4)& (1, 5)& (2,6)& (3,7)& (4,8)& (5,9) \\ \hline
 
 (13, 4)&(14, 5)& (15, 6)& (1,7)& (2,8)& (3,9)& (4,10) \\ \hline
 
 (12, 5)&(13,6)& (14, 7)& (15,8)& (1,9)& (2,10)& (3,11) \\ \hline
 
 (11,6)&(12,7)& (13, 8)& (14,9)& (15,10)& (1,11)& (2,12) \\ \hline
 
 (10, 7)&(11,8)& (12, 9)& (13,10)& (14,11)& (15,12)& (1,13) \\ \hline
 
 (9,8)& (10,9) & (11, 10)& (12,11)& (13,12)& (14,13)& (15,14) \\ \hline
\end{tabular}

\hspace*{-0.9cm}
\begin{tabular}[t]{|C|C|C|C|C|C|C|}
\hline
 8 &9 & 10& 11&12 &13 &14 \\ \hline
(7,9) &(8,10) &(9,11) &(10,12) &(11,13) &(12,14) & (13,15)\\ \hline
 
(6,10)& (7,11)& (8,12)& (9,13)&(10,14) &(11,15) &(12,1) \\ \hline
 
(5,11)& (6,12)& (7,13)&(8,14) &(9, 15) &(10,1) &(11,2) \\ \hline
 
 (4,12)& (5,13)& (6,14)&(7,15) &(8,1) &(9,2) &(10,3) \\ \hline
(3,13)& (4,14)& (5,15)&(6,1) &(7,2) &(8,3) &(9,4) \\ \hline
 
(2,14)& (3,15)& (4,1)&(5,2) & (6,3)&(7,4) & (8,5) \\ \hline
 
(1,15)& (2,1)& (3,2)& (4,3)& (5,4) & (6,5) & (7,6) \\ \hline
\end{tabular}

}
\end{table}

The set of uncolored edges (using the same notational shortcut as before) is $\mathcal{M} := \{(14,1), (13,2), (12,3), (11,4), (10,5), (9,6), (8,7) \}$. 

Clearly, $K_{15} - \mathcal{M}$ is $14$-colorable via the table above. But note that our coloring attributes colors to edges that exist in $K_{15}$ but are missing from $\mathscr{G}(C_{15})$, namely $(3,5), (3,10), (6,5), (6,10), \\ (9,5), (9, 10), (12,5), (12,10)$. Call this set of edges $\mathcal{A}$. Following Rhee's construction, we will replace each edge from $\mathcal{A}$ by an edge from $\mathcal{M}$, until $\mathcal{M}$ is empty, in a way that keeps our graph $14$-colorable. 

An example case of the construction of Rhee was given above. For our specific case, starting with Table 2 and changing edges and colors between $\mathcal{A}$ and $\mathcal{M}$, there are two cases even before entering the construction. Either we are replacing an edge of $\mathcal{A}$ with an edge from $\mathcal{M}$ which is its ``conjugate" (following the terminology of Rhee), i.e., an edge missing from one of the two vertices incident to the existing edge. In this case we directly follow the construction. Otherwise, we need some intermediate steps until we find ourselves in a position where we are replacing an edge with one of its conjugates. 

For $\mathscr{G}(C_{15})$ a first step in the process could be to replace $(5, 6)$ - which shouldn't exist in $\mathscr{G}(C_{15})$, with $(5, 10)$ which should but is not in our current graph. To do so, similarly as before, consider the $\{13, 10\}$-color-alternating path starting at $10$ ($c^{10}$). It consists of vertices $10, 1, 4, 7, 13$. We can then invert each color in that path to make it a $\{10, 13\}$-color-alternating path. So $10$ does not have an edge colored $13$ incident to it, hence we can delete the edge $(5, 6)$, which was colored $13$, and create the edge $(5, 10)$ with that color. The resulting coloring table shows that the construction did not create any conflicts.

\begin{table}[H]
\caption{Coloring table after removing $(5, 6)$ and adding $(5, 10)$. }
{\small
\centering
\hspace*{-0.9cm} 
\begin{tabular}[t]{|C|C|C|C|C|C|C|}
\hline
 1& 2& 3& 4& 5& 6&7  \\ \hline
 (15, 2)&(1, 3) & (2, 4) & (3,5)& (4,6)& (5,7)& (6,8)\\ \hline
 
 (14, 3)&(15, 4)& (1, 5)& (2,6)& (3,7)& (4,8)& (5,9)\\ \hline
 
 (13, 4)&(14, 5)& (15, 6)& (1,7)& (2,8)& (3,9)& (4,10)\\ \hline
 
 (12, 5)&(13,6)& (14, 7)& (15,8)& (1,9)& (2,10)& (3,11) \\ \hline
 
 (11,6)&(12,7)& (13, 8)& (14,9)& (15,10)& (1,11)& (2,12) \\ \hline
 
 (10, 7)&(11,8)& (12, 9)& (13,10)& (14,11)& (15,12)& (1,13) \\ \hline
 
 (9,8)& (10,9) & (11, 10)& (12,11)& (13,12)& (14,13)& (15,14) \\ \hline
\end{tabular} 

\hspace*{-0.9cm} 
\begin{tabular}[t]{|C|C|C|C|C|C|C|}
\hline
8 &9 & 10& 11&12 &13 &14 \\ \hline
  (7,9) &(8,10) &(9,11) &(10,12) &(11,13) &(12,14) & (13,15)\\ \hline
 
  (6,10)& (7,11)& (8,12)& (9,13)&(10,14) &(11,15) &(12,1) \\ \hline
 
  (5,11)& (6,12)& \textcolor{blue}{(10, 1)} &(8,14) &(9, 15) & \textcolor{blue}{(5, 10)} &(11,2) \\ \hline
 
  (4,12)& (5,13)& (6,14)&(7,15) &(8,1) &(9,2) &(10,3) \\ \hline
 
  (3,13)& (4,14)& (5,15)&(6,1) &(7,2) &(8,3) &(9,4) \\ \hline
 
  (2,14)& (3,15)& \textcolor{blue}{(4, 7)} &(5,2) & (6,3)& \textcolor{blue}{(1, 4)} & (8,5) \\ \hline
 
  (1,15)& (2,1)& (3,2)& (4,3)& (5,4) & \textcolor{blue}{(7, 13)} & (7,6) \\ \hline
\end{tabular} 

}
\end{table}

\medskip

Now, another edge from $\mathcal{M}$ we want to add to our graph is, for example, $(7, 8)$, but we may not exchange it with any of its conjugates. In such cases, one is required to take multiple steps, such as removing $(5, 8)$ to add $(8, 7)$, and then adding back $(5, 8)$ instead of $(3, 5)$ which is one of the edges we do not need. These rely on a different case of the construction, one in which after changing the color of the alternating path, we need to handle a conflict that arises in a cycle. We shall not go further here as the color alternating path involved in exchanging $(5, 8)$ and $(8, 7)$ contains $14$ edges which will hence all change color.

\end{example}

\pagebreak

\section*{Acknowledgments}
The author greatly thanks Dominique Perrin for his help in reflecting on the problems, questioning the literature, and writing and structuring this paper. The author also thanks Peter Cameron for allowing him to discover the topic of graphs defined on groups and the problems around them. Finally, the author thanks the reviewers for their visual, typographical, and mathematical recommendations and corrections which helped improve the paper greatly.

\printbibliography

\end{document}